\documentclass[12pt,reqno,a4paper]{amsart}
\usepackage{amssymb,amsmath,amsthm,amscd}
\usepackage{mathtools}
\usepackage{esint}
\usepackage{mathrsfs}
 
\usepackage{mathabx}
\usepackage[T1]{fontenc}
\usepackage{hyperref}
\hypersetup{colorlinks,linkcolor={red},citecolor={blue}} 
\usepackage{enumerate}
\usepackage{tikz}
\usepackage{comment}

\newtheorem{theorem}{Theorem}
\newtheorem{lemma}[theorem]{Lemma}

\newtheorem{definition}[theorem]{Definition}

\newtheorem{proposition}[theorem]{Proposition}

\numberwithin{theorem}{section}
\numberwithin{equation}{section}

\def\N {{\mathbb N}}
\def\Z {{\mathbb Z}}

\def\R {{\mathbb R}}

\def\D {{\mathcal D}}

\def\<{\left\langle}
\def\>{\right\rangle}

\providecommand{\no}[1]{  \lVert  #1  \rVert }

\providecommand{\nos}[1]{  \lVert  #1  \rVert }

\title[Phase space localizing operators]{Phase space localizing operators}

\author{Marco Fraccaroli}
\author{Olli Saari}
\author{Christoph Thiele}

\address{BCAM - Basque Center for Applied Mathematics, 
    Mazarredo, 14 E48009 Bilbao, 
    Basque Country, Spain.}
\email{mfraccaroli@bcamath.org}

\address{Departament de Matem\`atiques, 
	Universitat Polit\`ecnica de Catalunya,
	Avinguda Diagonal 647, 08028 Barcelona,
	Catalunya, Spain}
\address{Centre de Recerca Matem\`atica, Edifici C, Campus Bellaterra, 08193 Bellaterra, Catalunya, Spain}
\email{olli.saari@upc.edu}

\address{Mathematical Institute, 
	University of Bonn,
	Endenicher Allee 60, 53115 Bonn,
	Germany}
\email{thiele@math.uni-bonn.de}

\subjclass[2020]{42B15, 42C15}
\keywords{Phase space localization, time-frequency analysis, modulation invariant operators, uniform estimates}

\begin{document}
	
\date{\today}

\begin{abstract}
We construct phase space localizing operators in all dimensions.
These are frequency localized variants of the conditional expectation operator related to a dyadic stopping time.
Our construction is an improvement over the so-called
phase plane projections of Muscalu, Tao and the third author in one dimension.
The motivation for such operators comes from time-frequency analysis. 
They are used in particular to prove uniform estimates for multilinear modulation invariant operators.
\end{abstract}

\maketitle

\section{Introduction}

Given a dyadic cube $U$ and a finite partition $\mathcal{P}$ of $U$ into smaller dyadic cubes, 
the corresponding conditional expectation operator is the orthogonal projection in $L^2(U)$ onto the subspace of functions measurable 
with respect to the sigma algebra $\Sigma_{\mathcal{P}}$ generated by $\mathcal{P}$. 
The conditional expectation $g$ of a function $f\in L^2(U)$ can be written as
\begin{align}\label{e:condexp}
g &=\sum_{I\in \mathcal{P}}A_If, 
\\
\label{e:cubeave}
A_If(x) &=1_I(x) \frac 1{|I|}\int_I f(y)\, dy.
\end{align}
The conditional expectation $g$ satisfies
\begin{equation}\label{e:projin}
A_I(f-g)=0\end{equation}
whenever a dyadic cube $I$ is in $\Sigma_{\mathcal{P}}$ and
\begin{equation}\label{e:projout}
1_Ig -A_Ig=0
\end{equation}
whenever $I$ is not in $\Sigma_P$.
If we define
\begin{equation}\label{e:dyadicS}
S:=\sup_{I\in \Sigma_P} \|A_If\|_{\infty},\end{equation} then the conditional expectation $g$ satisfies
\begin{equation}\label{e:projinf}
\|g\|_\infty \le  S.
\end{equation}

In the context of Fourier analysis, 
the sharp cutoffs in the definition of $A_I$ destroy frequency 
information and are undesirable. 
Instead of $A_I$ as above, one works with
phase space localized operators
of the type
\begin{equation}\label{e:convtq1}
A_If = \rho_I^{-\alpha} \times  (\phi_{|I|}*f).
\end{equation}
Here $\rho_I$ is a weight growing with relative distance to $I$, $\alpha>0$ is sufficiently large, and $\phi_{|I|}$ is a convolution kernel whose Fourier transform is supported in the ball of radius $|I|^{-1}$ about the origin. 
There is longstanding interest in constructing operators with similar properties as the conditional expectation 
but based on the smoother operators \eqref{e:convtq1}.
The obvious attempt to define a map from $f$ to $g$ as in \eqref{e:condexp} 
but with the operators $A_I$ as in \eqref{e:convtq1}, an operator that is sometimes used in similar context, for example in slightly modified form in \cite{dipliniofragkos},  fails for our purpose.
Such a map does not satisfy strong enough replacements for 
\eqref{e:projin}, \eqref{e:projout},
\eqref{e:dyadicS} and \eqref{e:projinf}.

The purpose of the present paper is to construct in Theorem \ref{thm:main_unippp} a function $g$
based on a partition $\mathcal{P}$ and a function $f$
so that suitable modifications of \eqref{e:projin}, \eqref{e:projout},
\eqref{e:dyadicS} and \eqref{e:projinf} hold, namely, in this order, 
\eqref{e:unif-gtheo}, \eqref{e:unigtheo-variant},
\eqref{e:defs} and \eqref{e:unifgnorm}. 
In place of unrealistically strong vanishing properties as in \eqref{e:projin}, \eqref{e:projout}, one has very strong estimates \eqref{e:unif-gtheo} and \eqref{e:unigtheo-variant} in terms of estimates at least as strong as Carleson measure estimates.
Such Carleson measure estimates are also called
sparseness or outer $L^\infty(\ell^1)$ bounds (\cite{MR3312633}, \cite{MR4292789})
in different parts of the literature. 
Note that we allow for a ratio $2^m$ between $|I|$ 
and the scale of $\phi$ in \eqref{e:convtq1}.
Uniformity of our estimates in the parameter $m$ is important for 
applications to uniform estimates for multilinear forms.
A precedent and motivating example for our phase space localizing operators
appears as so-called phase plane projections in \cite{MR1979774} in dimension $d=1$. 
We generalize the result to higher dimensions, 
simplify the construction, and strengthen the statement.

To state the main result in detail,
we fix a dimension $d\ge 1$. 
For a finite axis-parallel cube $I\subset \R^d$ and $r\ge 1$, 
let $r I$ denote the axis-parallel cube with the same center but $r$ times the side length.
Define for a point $y\in \R^d$ and a Borel set $F\subset \R$ the mollified distance
\[ \rho_I(y)=\inf\{r>1: y\in (2r-1)I\}, \quad \rho_I(F)= \inf_{ y\in F}\rho_I(y).\]
For an integer $j$ and a real number $\alpha>d$, let $\Phi_j^\alpha$ be the set of  continuous functions $\phi$ on $\R^d$ with
 \[|\phi|\le 2^{-d j}\rho_{[0,2^j]^d}^{-\alpha}\] and $\widehat{\phi}(\xi)=0$ if $|\xi|\ge 2^{-j}$,
where the Fourier transform is defined as
\[\widehat{\phi}(\xi)=\int_{\R^d} \psi(x)e^{-2\pi i x\cdot \xi}\, dx.\]
Let $\Psi_j^\alpha$ be the set of functions $\psi\in \Phi_j^\alpha$ with $\widehat{\psi}(\xi)=0$ for $|\xi|\le 2^{-j-2}$.

For $j\in \Z$ and $k\in \Z^d$, 
let $Q_{j,k}$ be the dyadic cube consisting of all points $y\in \R^d$ 
with $y_n\in [2^jk_n, 2^j (k_{n}+1))$ for $1\le n\le d$, where $y_n$ and $k_n$ are the components of
$y$ and $k$ respectively.
Let $\D_j$ be the set of cubes $Q_{j,k}$ with $k\in \Z^d$.
Let $\D=\bigcup_{j\in \Z} \D_j$ be the set of all dyadic cubes.
For a finite collection $M$ of pairwise disjoint cubes contained in a cube $U$, 
let $M_U$ be the set of cubes $J\in \D$ 
such that there exists $I \in M$ with $I\subset J\subset U$.

\begin{theorem}[Main theorem]\label{thm:main_unippp} 
Let $1 \le p \le \infty$ and  $1/p+1/p'=1$.
Let $\alpha>d$ be real.
There exists $C=C_{p,d,\alpha}>0$ such that for all $m \in \N$ the following holds.

Let $i_0\in \Z$ and $U
\in \D_{i_0}$. Let $M$ be a finite non-empty collection of pairwise disjoint cubes contained in $U$.
There exists a linear mapping sending each locally bounded measurable function 
$f$ on $\R^d$ with $\no{f\rho_{U}^{-\alpha}}_\infty<\infty$
to a Borel function $g$ supported in $5U$ with the following properties.
Denote 
\begin{equation}\label{e:defs}
S \coloneq \sup_{i\in \Z}\ \sup_{I\in \D_i\cap  {M}_U}\  \sup_{\phi \in \Phi_{i-m-2}^{4\alpha}} 2^{-id/p}\no{\rho_I^{-\alpha} \phi *f}_p.
\end{equation}
Then
\begin{equation}
\label{e:unifgnorm}
\no{g}_p \le C S 2^{i_0d/p},
\end{equation} 
and for every $j\le i_0$ and every $J\in \D_j$ it holds
\begin{equation}   
\label{e:unif-gtheo}
  \sum_{i \le i_0} \sum_{\substack{ I\in \D_i\cap {M}_U\\  I\subset J}}\  \sup_{\phi\in \Phi_{i-m}^{4\alpha}}
2^{id/p'}\no{\rho_I^{-3\alpha} \phi *(f-g)}_p \le CS2^{jd},  
\end{equation}
and if no $J'\in \D_j$ 
with $\rho_J(J')\le 1$ contains an element of 
$M_U$, it holds
\begin{equation}
\label{e:unigtheo-variant}
\sum_{i \le i_0} \sup_{\substack{I \in \D_i \setminus {M}_U  \\  I\subset J}}\ 
\sup_{\psi\in \Psi_{i-m}^{4\alpha}}
2^{-id/p}\no{\rho_I^{-3\alpha} \psi *g }_p  \le CS\no{1_U \rho_J^{-\alpha}}_\infty.
\end{equation}
\end{theorem}

To emphasize dependence on parameters, 
denote the quantity in \eqref{e:defs} by $S_{\alpha,p}$.
One obtains estimates similar to \eqref{e:unifgnorm}, \eqref{e:unif-gtheo} and \eqref{e:unigtheo-variant}
with $S_{\alpha,p}$ replaced by $S_{\beta,q}$ with different parameters by using the following easy proposition.
\begin{proposition}\label{t:spq}
Let $1\le p,q\le \infty$ and $\alpha>d$. 
There exists $C>0$ such that the following holds. 
Let $i_0$, $U$, $M$ $f$,  $m$ as in Theorem \ref{thm:main_unippp}.
If $p\le q$, then
\begin{equation}\label{e:holder}
S_{\alpha(1+\frac {q-p}{qp}),p}\le CS_{\alpha,q}.
 \end{equation}
If $q\le p$, then
\begin{equation}\label{e:logconvex}
   S_{\alpha,p}\le C S_{\alpha,q}^{q/p}
\|f \rho_{U}^{-\alpha}\|_\infty^{1-q/p}
 \end{equation}
and
\begin{equation}\label{e:bernstein}S_{\alpha,\infty}\le C 2^{dm} S_{\alpha,1}.\end{equation}
\end{proposition}
The estimate \eqref{e:logconvex}
loses control in the size by virtue of the power $q/p$ and is thus more useful for $q/p$ close to one.
The estimate \eqref{e:bernstein}
loses uniformity in the parameter $m$ and is more useful for small values of $m$.
Using Proposition \ref{t:spq} together with Theorem \ref{thm:main_unippp},
we recover all the estimates that one might want to infer from Proposition 7.4 of \cite{MR1979774}.
We emphasize that the square sum term in the definition of the size in that paper
is not needed to estimate the phase space localization errors in Theorem \ref{thm:main_unippp}.
 
Concerning the technical aspects of the present paper,
our construction starts with a telescoping sum involving characteristic functions of cubes as in \cite{MR1979774}.
We then correct the jump discontinuities appearing in the construction 
with a family of coronas based on smooth restrictions of certain primitives of Littlewood--Paley pieces. 
The correction part constructed here is different from that in \cite{MR1979774} 
and more amenable to our setup.
It provides vanishing of higher moments, which is needed when $d > 1$. It also provides
control over derivatives of the phase space projection, which is needed to prove \eqref{e:unigtheo-variant} when $d > 1$.
As a consequence,
we can treat general trees with a unified construction without 
distinguishing lacunary and non-lacunary trees
 as in \cite{MR1979774}.
In dimension $d=1$, 
our construction provides a simple and streamlined substitute to \cite{MR1979774}.

Concerning applications of our main theorem, we recall that the precedent for our main theorem, the one-dimensional phase plane projections appearing in \cite{MR1979774},
was used to prove uniform estimates for multilinear singular integrals with modulation invariance. In such modulation invariant setting, one typically conjugates a Theorem like
\ref{thm:main_unippp} with modulations.
In more detail, one applies the theorem to the functions $M_{-\eta} f$,
where the modulation $M_\eta$ is defined by
\[M_\eta f(x):=f(x) e^{2\pi i x\cdot \eta} .\]
Let $g_1$ and
$g_2$ be obtained by applying the theorem
to $M_{-\eta_1}f$ and $M_{-\eta_2}f$ respectively with possibly different collections of cubes .
If $\eta_1$ and $\eta_2$ are sufficiently 
far from each other,
measured in terms of the geometry of the cubes,
then
$M_{\eta_1}g_1$ and 
$M_{\eta_2}g_2$
are almost orthogonal
to each other thanks to the good frequency
localization in Theorem \ref{thm:main_unippp}. 
This is crucial in the
application to modulation invariant operators.
In the companion paper \cite{fst-uniform}, we use Theorem \ref{thm:main_unippp} to obtain uniform bounds for multilinear singular integrals with modulation symmetries in the plane and space.

For further background on related uniform bounds in $d=1$, 
we refer to \cite{MR1933076}, \cite{MR2113017}, \cite{MR2320411}, \cite{MR3538147} and \cite{gennady_uniform} for uniform bounds for bilinear Hilbert transforms in various regions of exponents and 
to \cite{MR2701349}, \cite{MR2197068}, \cite{MR3000982}, \cite{MR2413217} and \cite{saarithiele} on bilinear multipliers closely related to uniform bounds for the bilinear Hilbert transform.
We also point out that the classical dyadic conditional expectation is used similarly together with
Walsh--Fourier modulations in the context of Walsh models for multilinear singular integrals; 
see for example \cite{MR2997005} and \cite{MR3482272}.
Finally, 
we mention that phase plane projections in time-frequency analysis are not only 
utilized for establishing uniform bounds; see \cite{MR4403962} for an application to vector valued estimates.

\bigskip

\noindent \textit{Acknowledgement.} 
The authors were funded by the Deutsche Forschungsgemeinschaft (DFG, German Research Foundation) 
under the project numbers 390685813 (EXC 2047: Hausdorff Center for Mathematics) and 211504053 (CRC 1060: Mathematics of Emergent Effects).
The first author was supported by the Basque Government through the BERC 2022-2025 program and by the Ministry of Science and Innovation: BCAM Severo Ochoa accreditation CEX2021-001142-S / MICIN / AEI / 10.13039/501100011033.
The second author was supported by Generalitat de Catalunya (2021 SGR 00087), 
Ministerio de Ciencia e Innovaci\'on and the European Union -- Next Generation EU (RYC2021-032950-I),  (PID2021-123903NB-I00), the Spanish State Research Agency 
through the Severo Ochoa and María de Maeztu Program for Centers and Units of Excellence in 
R\&D (CEX2020-001084-M).

\section{Construction of the phase space projection}

Let $\alpha>d$.
The letter $C$ will denote a sufficiently large positive number 
that may be implicitly re-adjusted from inequality to inequality 
and that may depend on $\alpha$ and $d$. 
We denote the Lebesgue measure of a measurable set $E$ by $|E|$.
For a multi-index $l \in \N^{d}$,
we denote 
$$|l| = \sum_{n=1}^d l_n, \quad l! = \prod_{n=1}^d l_n! , $$
and for $k,l \in \N^{d}$, we write $k \le l$ if $k_n \le l_n$ for all $1\le n \le d$.

Let $m\in \N$.
By dilation and translation that observe the dyadic grid, it suffices to prove
Theorem \ref{thm:main_unippp} for $i_0=0$ and $U$
the unit cube $[0,1)^d$, a normalization that we henceforth assume. 
Let $M$ and $f$ be given as in the theorem.
Define for $j\in \Z$
\[T=M_U, \quad T_j^0\coloneq  T\cap \D_j, \quad  E_j^0\coloneq 
{\bigcup T_j^0}.\]
Define further for an integer $k \ge 1$
\[
T_j^k\coloneq\{I\in \D_{j}: \rho_I(E_j^0) \le k\}, \quad
E_j^{k}\coloneq{\bigcup T_j^k}, 
\]
and for an integer $k \ge 0$
\[
  B_j^k\coloneq  T_j^{k+1}\setminus T_j^k .
\]
Define $F_j^k$ to be $\R^d\setminus E_j^k$.

Let
$\tau$ be a Schwartz function on $\R^d$ such that 
$\widehat{\tau}(\xi)=0$
if $|\xi|\ge 2$ and $\widehat{\tau}(\xi)=1$
if $|\xi|\le 1$. Consider the  approximation of unity formed by the family of convolution kernels
\[\tau_j(x)=2^{-jd}\tau(2^{-j}x).\]
For the purpose of telescoping, define the difference kernels
\[\psi_j=\tau_{j-1}-\tau_{j}.\]
We observe $\widehat{\psi}_j(\xi)=0$ if $|\xi|\ge 2^{2-j}$ or $|\xi|\le 2^{-j}$. 

\subsection{Cone decomposition}
\label{sec:cone-decomposition}
We decompose $\psi_j$ using the open cones
\[O_n=\{\xi\in \R^d:2d|\xi_n|^2>|\xi|^2\}\]
for $1\le n\le d$. 
Any point $\xi\in \R^d$ not covered by any of the cones satisfies
\[d|\xi|^2=\sum_{n=1}^d |\xi|^2\ge \sum_{n=1}^d 2d|\xi_n|^2= 2d|\xi|^2,\]
and thus $\xi=0$. It follows that there is for each $n$ a smooth
function $\widehat{\chi}_n$ with $\widehat{\chi}_n(\xi)=0$ if $\xi\notin O_n$
such that 
\[\sum_{n=1}^d \widehat{\chi}_n(\xi)=1\]
if $1\le |\xi|\le 4$.
Define
\[\chi_{n,j}(x)=2^{-dj}\chi_n(2^{-j}x), \quad \psi_{n,j}=\psi_j*\chi_{n,j}.\]
In particular, we observe
\[\psi_j=\sum_{n=1}^d \psi_{n,j}.\]

\subsection{Construction}
\label{sec:construction}

The function $\widehat{\chi}_n$
vanishes on the set $2d|\xi_n|^2\le 1$. 
Hence we find Schwartz functions $\theta_{n,j}$ that satisfy 
\[\partial_n^{d+1} \theta_{n,j}=\psi_{n,j}.\]
The function $\widehat{\theta}_{n,j}$ has the same support as $\widehat{\psi}_{n,j}$, 
in particular $\widehat{\theta}_{n,j}=0$ unless $2^{-j}\le |\xi|\le 2^{2-j}$.
Let $\kappa$ be a smooth function mapping  $\R^d$ to $[0,\infty)$ with $\kappa(x) = 1$ for $|x| \le 2^{-10}$ and
$\kappa(x) = 0$ for $|x| \ge 2^{-9}$.
For $j \in \Z$ and $j \le 0$,
we set $\kappa_j(x) = 2^{-jd}\kappa_j(2^{-j}x)$.
Define 
\begin{equation}
\label{e:deriv-of-sigma}
\sigma_{n,j} := (- 1_{E_j^{1}}+\kappa_{j-m} * \sum_{I\in \D_{j-m-3}, \rho_I(E_j^1)\le 2}1_{I} ) \times (\theta_{n,j-m}*f).
\end{equation} 
The function $\sigma_{n,j}$ is supported on the set of points which have positive distance
at most $2^{j-m-1}$ from $E_j^1$, in particular it is supported on $E_j^2$ and vanishes on $E_j^1$.

The following lemma follows immediately from the construction.


\begin{lemma}\label{l:contderivative}
For each non-positive $j \in \Z$, $1\le n \le d$, the function
\begin{equation}\label{e:lemfct}
G_{n,j}:=(\theta_{n,j-m}*f)(x) 1_{E_j^1} + \sigma_{n,j}
\end{equation}
is infinitely differentiable in $\R^{d}$.
For a multi-index $\beta\in \N^d$,  one has for almost almost every $x$
\begin{equation}\label{e:derlemfct}
\partial^\beta G_{n,j}(x)=(\partial^\beta \theta_{n,j-m}*f)(x) 1_{E_j^1}(x) + \partial^\beta\sigma_{n,j}(x),
\end{equation}
where $\partial^\beta\sigma_{n,j}$ is the classical derivative at points 
outside the boundary of $E_j^1$
and remains undefined at points of the boundary of $E_j^1$.
\end{lemma}
  
Define 
\begin{equation}
\label{e:gnj_def}
g_{n,j}:=   \partial_n^{d+1} G_{n,j}.
\end{equation}
Finally, 
we set 
\[\chi = \kappa_{-m} * \sum_{I\in \D_{-m-3}, \rho_I(E_0^1)\le 2}1_{I} , \]
so that $\chi(x) = 1$ for $x \in E_0^{1}$ and $\chi(x) = 0$ for all $x \in F_0^2$.


Define
\begin{equation}
\label{e:gnj_deco}
g:=(\tau_{-m}*f)\chi + \sum_{n=1}^d \sum_{j\le 0} g_{n,j}.
\end{equation}
In particular, $g$ is zero outside $5U$.

We note for later use that
\begin{align}
\label{tauphi}
\tau_j &\in C\Phi_{j-1}^{4\alpha},\\
\label{psipsi}
\psi_{n,j}& \in C{\Psi}_{j-2}^{4\alpha},\\
\label{e:thetaphi}
 \partial^{l} \theta_{n,j}&\in C 2^{j(d+1-|l|)} \Psi_{j-2}^{4\alpha}, \quad |l| \le 3d+3, \\
\label{kappaineq}
\|\partial^l\chi\|_\infty &\le C 2^{m|l|}, \quad |l|\le 3d+3,
\end{align} 
for some $C$ depending on the above choices of the Schwartz functions. 
We shall assume these functions are chosen so as to nearly minimize this $C$, 
hence $C$ can be considered depending only on $d$ and $\alpha$.

\section{Proof of inequality
\texorpdfstring{\eqref{e:unifgnorm}}{(\ref{e:unifgnorm})}}

The construction of $g$ was independent of $p$.
To avoid notational differences between finite and infinite $p$ in this section, 
we shall assume $1\le p<\infty$ and prove
bounds uniformly in $p$. 
The case $p=\infty$ then follows by a limiting process as $p$ tends to $\infty$.

We begin the proof  of \eqref{e:unifgnorm}  by writing 
\[g=h+\sum_{n=1}^d k_n\] 
with
\begin{align}
\label{def:hpart}
h &\coloneq (\tau_{-m}*f) \chi +
\sum_{j\le 0}(\psi_{j-m}*f)1_{E_j^1}, \\
\label{def:knpart}
k_n &\coloneq 
\sum_{j\le 0}  \partial_n^{d+1} \sigma_{n,j} . 
\end{align}
By the triangle inequality, 
it suffices to prove $\no{h}_p\le CS$ and  $\no{k_n}_p\le CS$ for each $n$ separately. We begin with $h$.

\begin{definition}
Let $j<  0$ be an integer. 
We define $\mathcal{I}_j$ to be the set of dyadic cubes in $\D_j$ 
which are contained in $E_{j+1}^1$ but not contained in $E_{j}^{1}$. 
We set $\mathcal{I} =\bigcup_{i < 0} \mathcal{I}_i$.
\end{definition}
 
\begin{lemma}
\label{l:Iispartition}
The family $\mathcal{I}$ partitions $E_0^1$. 
\end{lemma}
\begin{proof}  
Let $x\in E_0^1$. 
By nesting of the sets $E_j^1$ and finiteness of $M$, 
there is a minimal $j < 0$ such that $x \in E_{j+1}^1$. 
The cube $I\in \D_j$ containing $x$ is then in $\mathcal{I}_j$. 
Hence $\mathcal{I}$ covers $E_0^1$.

To see that the cubes in $\mathcal{I}$ are pairwise disjoint,
assume to get a contradiction that $I \subset J$ for some $i<j<0$ 
and $I \in \mathcal{I}_{i}$ and $J \in \mathcal{I}_{j}$. 
Then $I\subset E_{i+1}^1$ and thus $I\subset E_{j}^{1}$. 
By the dyadic structure, $J\subset E_{j}^{1}$. 
This contradicts the choice of $J$.
\end{proof}

\begin{lemma}
\label{l:hpart}
Let $I \in \mathcal{I}$.
Then 
\begin{equation} 
\label{e:hpart}
\no{h1_I}_p\le C S |I|^{1/p}.
\end{equation}
\end{lemma}
\begin{proof}
Fix $i< 0$ and $I\in \mathcal{I}_i$.
As $I \subset  E_{i+1}^1$ but $I$ is disjoint from $E_{i}^{1}$, 
we have by telescoping
\[h1_I= (\tau_{i+1-m}*f)1_I.\]
Moreover, 
there is a $K \in T_{i+1}^0$ with $\rho_K(I)\le 2$. 
We have
\[\tau_{i+1-m}\in C\Phi^{4\alpha}_{i-m},\]
and thus by the definition of $S$
in \eqref{e:defs},
\[ \no{h1_I}_p\le C\no{\rho_K (\tau_{i+1-m}*f)}_p\le CS2^{id/p}.\]
\end{proof}

As $\mathcal{I}$ is a partition of $E_0^1=3U$ by Lemma \ref{l:Iispartition} 
and as $h$ is supported on $3U$,
we obtain with Lemma \ref{l:hpart}
\[
\no{ h }_p^p\le \sum_{I\in \mathcal{I}} \no{h1_I}_p^p \le C S
 \sum_{I\in \mathcal{I}} 2^{id} \le CS.
\]
This completes the proof of the bound for $h$.

It remains
to prove the bound for $k_n$.
Note that $\partial_n^{d+1} \sigma_{n,i}$ is supported in $\bigcup B_i^1$.
\begin{lemma}
\label{l:sigma-only}
For each $i \le 0$, $1\le n \le d$, $I\in B_i^1$ and each multi-index $l \in \N^{d}$ with $|l| \le 2d+2$
it holds
\begin{equation}
\label{e:sigmabound}
\no{1_{I} \partial^{l} \sigma_{n,i} }_p^{p}
  \le C^{p} 2^{p(i-m)(d+1-|l|)} S^{p} |I|.
\end{equation}
\end{lemma}
\begin{proof}
We define
\[\tilde{E}_i=\sum_{I\in \D_{i-m-3}, \rho_I(E_i^1)\le 2}1_{I} .\]
We apply the Leibniz rule in the interior of $I$
to estimate the left hand side of \eqref{e:sigmabound} by 
\[
C^{p} \sum_{l_1 + l_2 = l} \int_{I} |\partial^{l_1} (1_{\tilde{E}_i} * \kappa_{i-m})(x)|^{p}| \partial^{l_2} (\theta_{n,i-m} * f)(x)|^{p} \, dx.
\]
We estimate 
\[
1_{I}(x)|\partial^{l_1} (1_{\tilde{E}_i} * \kappa_{i-m})(x)|^{p} \le C 2^{p(m-i)|l_1|} 1_{I}(x)
\]
by a rescaled version of \eqref{kappaineq}.
Using the control \eqref{e:thetaphi} on $\partial^{l_2} \theta_{n,i-m}$,
we see that  
\[\int_{I} |\partial^{l_2} (\theta_{n,i-m}*f)(y)|^{p} \, dy \le 
C^p 2^{p(i-m)(d+1-|l_2|)} S^{p} ,
\]
where we also used the definition of $S$ for a cube $J \in T_i^0$ with
$\rho_J(I)\le 2$.
This concludes the proof of the lemma.
\end{proof}

\begin{lemma}
\label{l:knpart}
For $1 \le n \le d$ and $k_n$ from \eqref{def:knpart},
it holds
\begin{equation}
\label{e:knpart}
\no{k_n}_p \le CS.
\end{equation}
\end{lemma}
\begin{proof}
We estimate
\[
\no{k_n}_p^p \le \sum_{i\le 0} \sum_{I\in B_{i}^{1}}\no{1_{I} \partial_n^{d+1}\sigma_{n,i}}_p^p 
\le \sum_{i\le 0} \sum_{I\in B_{i}^{1}} C^p  S^p |I|
\]
where we applied Lemma \ref{l:sigma-only} with $|l| = d+1$. 
To conclude \eqref{e:knpart}, it now suffices to
show that the cubes in  $\bigcup_{i\le 0}B_i^1$
are pairwise disjoint and contained in $3U$.
Containment in $3U$ is clear by definition. 
To see disjointness, first note that the cubes in $B_i^1$ are pairwise disjoint for each $i$. 
Now let $i<k\le 0$ and assume $J\in B_i^1$ and $K\in B_k^1$ with $J\subset K$.
Then there is a point $x$ in the closure of $E_i^0$ such that $\rho_J(x)\le 1$. 
But then $\rho_K(x)\le 1$ and $x$ must be in the closure of $E_k^0$, 
which is a contradiction to $K \in B_k^1$.
This proves disjointness and completes to proof of the lemma.
\end{proof}

\section{Proof of Inequality \texorpdfstring{\eqref{e:unif-gtheo}}{(\ref{e:unif-gtheo})}}

In this section, we fix $1\le p\le \infty$. In particular, notation in this section applies to $p=\infty$ as well.

We begin with a preparatory lemma that allows us to commute a weight past certain convolution and truncation operators.
Let $1_F$ denote the characteristic function of a set $F$.

\begin{lemma}\label{l:commutechii}
For all $i\in \Z$,
all finite cubes $I\in \D_i$, 
all Borel sets $F \subset \R^{d}$, 
all functions $\phi$ with 
\[
|\phi(x)| \le 2^{(m-i)d} \rho_{[0,2^{i-m}]^{d}}(x)^{-4\alpha},
\]
and for all functions 
$g$ in $L^p(\R^{d})$,  
we have 
\begin{equation}\label{e:comwc}
 \no{\rho_I^{-3\alpha} (\phi* (1_F g))}_{p}
 \le 
C \no{\rho_I^{-\alpha} g}_{p} \rho_I(F)^{-2\alpha}.  
\end{equation}
\end{lemma}

We remark that given $k \le 0$ and $\phi \in \Phi_{i-m+k}^{4\alpha}$,
the function $2^{kd}\phi$ satisfies the assumption of the lemma.

\begin{proof}
By simultaneous translation of $I$, $F$, and $g$, if necessary,
 we may assume that $I$ contains the origin.
Let $J\in \D$ be a cube  in $\D_{i-m}$ containing zero,
so that
\[|\phi|\le C2^{(m-i)d}\rho_J^{-4\alpha}.\]

For every $x\in \R^d$, we estimate 
\begin{equation}
\label{e:bigint}
|\rho_I^{-3\alpha} (\phi* (1_F g))(x)|
\le 
C2^{(m-i)d}\int 
\rho_I^{-3\alpha}(x) \rho_J^{-4\alpha} (x-y) 
1_F(y) |g|(y)\, dy. \end{equation}
For $|y|\ge 2|x|$, we estimate 
\[\rho_I(x)\ge 1, \quad \rho_J(x-y)\ge C\rho_I(y).\]
For $|y|\le 2|x|$, we estimate 
\[\rho_J(x-y) \ge 1,\quad \rho_I(x)\ge C\rho_I(y).\]
We thus obtain for \eqref{e:bigint} an upper bound by
\[
C2^{(m-i)d}
\int 
 \rho_J^{-\alpha} (x-y) 
(\rho_I^{-2\alpha}1_F)(y) (\rho_I^{-\alpha}|g|)(y)\, dy. 
\]
Using Young's convolution inequality and the uniformly bounded $L^1$ norm of $2^{(m-i)d}\rho_J^{-\alpha}$,
one obtains the lemma.
\end{proof}

We write almost everywhere
\begin{equation}\label{fminusg}
f-g =(\tau_{-m} *f) (1-\chi)+   
  \sum_{j\le  0}
(\psi_{j-m}*f) 1_{F_j^1}
-   \sum_{n=1}^d  
\partial_n^{d+1} \sigma_{n,j},
\end{equation}
which follows from
\begin{align*}
f &=(\tau_{-m}*f) +\sum_{j\le  0} \psi_{j-m} *f,\\
g &= (\tau_{-m} *f) \chi+  \sum_{j\le  0} (\psi_{j-m}*f) 1_{E_j^1}+ \sum_{n=1}^d \partial_n^{d+1}\sigma_{n,j}.
\end{align*}

\begin{lemma}\label{l:f-gI1}
For all integers $i\le j\le  0$ and all $I\in T_i^0$, we have
\begin{align}
\label{e:ilessjtau}
\no{\rho_I^{-3\alpha} \phi_{i-m} *((\tau_{-m}*f) (1-\chi))}_p
&\le CS \rho_I(F_0^1)^{-2\alpha}, \\
\label{e:ilessj}
\no{\rho_I^{-3\alpha} \phi_{i-m} *(
(\psi_{j-m}*f) 1_{F_j^1})}_p
&\le CS2^{jd/p}\rho_I(F_j^1)^{-2\alpha}, \\
\label{e:ilessjchi}
\nos{\rho_I^{-3\alpha} \phi_{i-m} *
\left( 
\partial_n^{d+1}\sigma_{n,j}\right)}_p
&\le CS2^{jd/p}\rho_I(F_j^1)^{-2\alpha}.
\end{align}
The right-hand sides of both of the inequalities \eqref{e:ilessj} and \eqref{e:ilessjchi} are bounded by
\begin{equation}
\label{e:better}
CS2^{-id/p'} 2^{(\alpha-d/p)(i-j)}\no{1_I\sum_{K\in B_i^0} \rho_K^{-\alpha}}_1.
\end{equation}
\end{lemma}

\begin{proof} 
We begin with \eqref{e:ilessjtau}.
As $1-\chi$ is supported in $F_0^{1}$,
we use Lemma \ref{l:commutechii} with $\phi_{i-m}\in \Phi^{4\alpha}_{i-m}$ 
and then $I \subset U$ to estimate the left-hand-side of \eqref{e:ilessjtau} by
\[C\rho_I(F_0^1)^{-2\alpha}\no{\rho_I^{-\alpha} (\tau_{-m}*f)}_p 
\le C\rho_I(F_0^1)^{-2\alpha}\no{\rho_{U}^{-\alpha} (\tau_{-m}*f)}_p.\]
With the definition of $S$, using the control \eqref{tauphi} of $\tau_{-m}$ and $U\in T$, 
this proves  \eqref{e:ilessjtau}.

To estimate the left-hand-side of \eqref{e:ilessj},
choose a cube $J\in \D_j$ that contains $I$.
Using  Lemma \ref{l:commutechii} as above, 
we estimate the left-hand-side of \eqref{e:ilessj} by 
\[
C\rho_I(F_j^1)^{-2\alpha}\no{\rho_I^{-\alpha} (\psi_{j-m}*f)}_p \le 
C\rho_I(F_j^1)^{-2\alpha}\no{\rho_J^{-\alpha} (\psi_{j-m}*f)}_p .
\]
Using the control \eqref{psipsi} of $\psi_{j-m}$ and $J\in T$, this proves  \eqref{e:ilessj}. 

As $\sigma_{n,j}$ is supported in $E_{j}^{2} \cap F_{j}^{1}$,
we estimate the left hand side of \eqref{e:ilessjchi} again with Lemma \ref{l:commutechii} by 
a constant multiple of 
\[
 \rho_I(F_j^1)^{-2\alpha} \no{\rho_I^{-\alpha} \sum_{K\in B_j^1} 1_K
\partial_n^{d+1}\sigma_{n,j}}_p  
\le 
C S \rho_I(F_j^1)^{-2\alpha} 2^{jd/p}
\sum_{K\in B_j^1} \sup_{x\in K} \rho_I^{-\alpha}(x).
\]
In the last inequality, 
we used equation \eqref{e:sigmabound} from Lemma \ref{l:sigma-only} 
and that each $K\in B_j^1$ is contained in $F_j^1$.
This proves \eqref{e:ilessjchi}, 
because $i\le j$ and thus for every integer $n\ge 0$ 
there are at most $C(1+|n|)^{d-1}$ cubes $K\in B_j^1$ 
such that 
\[|n|+1\le \inf_{x\in K} \rho_I(x)\le |n|+2 ,\] 
and thus
\[\sum_{K\in B_j^1} \sup_{x\in K} \rho_I^{-\alpha}(x)\le  C\sum_{n\in \mathbb{Z} }
(1+|n|)^{d-\alpha-1}\le C.\]

We turn to showing the upper bound \eqref{e:better}.
By nesting, $\rho_I(F_{j}^{1}) \ge \rho_I(F_{i}^{1})$.
On the other hand, there is a cube in $B_j^0$ between $I$ and $F_j^1$.
Hence $\rho_I(F_{j}^{1})\ge 2^{j-i}$, 
and we estimate the  right-hand side of \eqref{e:ilessj} and \eqref{e:ilessjchi} by
\[ CS2^{jd/p} 2^{-\alpha (j-i)} \rho_I(F_i^1)^{-\alpha} \le 
CS2^{id/p}   2^{-(\alpha-d/p) (j-i)}\rho_I(F_i^1)^{-\alpha}.\]
As there is a cube in $B_i^0$ as close to $I$ as to $F_i^1$, 
we have
\[\rho_I(F_{i}^{1})^{-\alpha}
\le \inf_{x\in I}\sum_{K\in B_i^0} \rho_K(x)^{-\alpha}
\le 2^{-id}\no{1_I\sum_{K\in B_i^0} \rho_K^{-\alpha}}_1.\]
This proves \eqref{e:better} and completes the proof of the lemma.
\end{proof}

\begin{lemma}\label{l:f-gI3}
We have for all integers  $j \le i\le 0$,
every dyadic cube $I\in T_i^0$ 
and every $1\le n\le d$,
\begin{multline}
\label{smallj}
\nos{\rho_I^{-3\alpha} \phi_{i-m} *\left(
(\psi_{n,j-m}*f)1_{F_j^1} - 
\partial_n^{d+1}\sigma_{n,j}\right)}_p \\
\le
C S 2^{(j-i)(1+d/p)}
2^{-id/p'}\sum_{J\in B_j^0}
\no{\rho_I^{-\alpha}1_{J}\, }_1 .
\end{multline}
\end{lemma}

\begin{proof}
Note that the term in brackets on the left-hand-side of 
\eqref{smallj} has a smooth extension to $\R^d$ by
Lemma \ref{l:contderivative}.
By a $(d+1)$-fold integration by parts and 
the sufficient decay of $\partial_n^k\phi_{i-m}$ for all $1\le k\le d$, 
we rewrite and then estimate the left hand side of \eqref{smallj} as follows:
\begin{gather}
\nos{\rho_I^{-3\alpha} (\partial_n^{d+1} \phi_{i-m}) *\left(
(\theta_{n,j-m}*f)1_{F_j^1}- 
\sigma_{n,j}\right)}_p \nonumber \\
\label{firstprim}
\le 
\no{\rho_I^{-3\alpha} (\partial_n^{d+1} \phi_{i-m} *(
(\theta_{n,j-m}*f)1_{F_j^1}))}_p \\
\label{secondprim}
+\sum_{J\in B_j^1}\no{\rho_I^{-3\alpha} (\partial_n^{d+1} \phi_{i-m} *
(1_J\sigma_{n,j}))}_p.
\end{gather}

We first estimate \eqref{secondprim}.
As $\partial_n^{d+1} \phi_{i-m}\in 2^{(m-i)(d+1)}\Phi_{m-i}^{4\alpha}$,
we estimate each summand in \eqref{secondprim} first with Lemma \ref{l:commutechii}
and then apply equation \eqref{e:sigmabound} from Lemma \ref{l:sigma-only}
to obtain an upper bound by 
\begin{multline*}
C2^{(m-i)(d+1)} \no{\rho_I^{-\alpha}
1_J \sigma_{n,j}}_p  \\
\le
CS2^{(m-i)(d+1)}2^{(j-m)(d+1)}2^{jd/p}\no{\rho_I^{-\alpha}1_J}_\infty  \\ 
\le
CS2^{(j-i)(1+d/p)}2^{-id/p'}\no{\rho_I^{-\alpha}1_J}_1 .
\end{multline*}
We used $j<i$ in the last inequality to conclude that $\rho_I^{-\alpha}$ is essentially constant on $J$.
The desired estimate for \eqref{secondprim} follows from summing over  $J\in B_j^1$.

To estimate \eqref{firstprim}, 
we write it as 
\[\no{\rho_I^{-3\alpha} \partial_n^{d+1} \phi_{i-m} *\tau_{j-m}*(
(\theta_{n,j-m}*f)1_{F_j^1})}_p ,\]
because $\widehat{\tau}_{j-m+1}$ is constant one on the support of $ \widehat {\partial_n^{d+1}\phi}_{i-m}$
when $i\ge j+1$.
By Lemma \ref{l:commutechii},
this is estimated by 
\begin{multline*}  C 2^{(m-i)(d+1)}\no{\rho_I^{-\alpha}  (\tau_{j-m+1}*((\theta_{n,j-m}*f)1_{F_j^1}))}_p  \\
\le C2^{(m-i)(d+1)}\sum_{K\in \D_{j}} \rho_I^{-\alpha}(K) \no{1_K  (\tau_{j-m+1}*( (\theta_{n,j-m}*f)1_{F_j^1}))}_p .
\end{multline*}

Fix $K\in \D_j$. 
First assume $K\subset E_j^1$.
Because $1_{K} \le \rho_K^{-3\alpha}$,
it follows by Lemma \ref{l:commutechii} that  
\begin{multline*}\no{1_K  (\tau_{j-m+1}*(
(\theta_{n,j-m}*f)1_{F_j^1}))}_p  
\le C \rho_K(F_j^1)^{-2\alpha}\no{\rho_K^{-
\alpha}  
(\theta_{n,j-m}*f)}_p  \\
\le C S2^{(j-m)(d+1)}2^{jd/p}\rho_K(F_j^1)^{-2\alpha} 
\le C S2^{(j-m)(d+1)}2^{jd/p}
\sum_{J\in B_j^1}\rho_J(K)^{-2\alpha}.
\end{multline*}
The penultimate inequality followed by $K\subset E_j^1$,
and thus there existing a cube $K'\in T_j$ with distance at most $2^{j+1}$ from $K$, 
and applying the definition of $S$ to $K'$.
The last inequality followed by 
choosing a $J$ in $B_j^1$ closest to $K$ so that
\[\rho_K(F_j^1)^{-2\alpha}\le C\rho_{K}(J)^{-2\alpha}= C\rho_{J}(K)^{-2\alpha},\]
where the last identity follows from equal side length of $J$ and $K$.

Assume then $K\subset F_j^1$.
As $\widehat{\theta}_{n,j-m}$ vanishes on the support of $\widehat{\tau}_{j-m+1}$ 
and $1_{F_j^1}+ 1_{E_j^1}=1$, 
we have 
\begin{multline}\no{1_K  (\tau_{j-m+1}*(
(\theta_{n, j-m}*f)1_{F_j^1}))}_p =\no{1_K  (\tau_{j-m+1}*(
(\theta_{n, j-m}*f)1_{E_j^1}))}_p \\
\label{e:aux43}
\le C \rho_K(E_j^1)^{-2\alpha}
\no{\rho_K^{-1}(\theta_{n, j-m}*f)}_p
\end{multline}
Now let $K'$ be a cube in
$B_j^0$ closest to $K$, so that
\[\rho_K(E_j^1)^{-1}\le C\rho_K(K')^{-1}.\]
With the triangle inequality, we conclude
\[\rho_{K'}(x)\le C\rho_{K}(x)\rho_{K}(E_j^1)\]
so that 
\[
\rho_{K}(x) \ge \rho_{K'}(x) \rho_{K}(E_j^1)^{-1} \ge \rho_{K'}(x) \rho_{K}(E_j^1)^{-\alpha}
\]
and hence we can estimate \eqref{e:aux43} by
\[  \rho_K(K')^{-\alpha}
\no{\rho_{K'}^{-1}(\theta_{n,j-m}*f)}_p \le CS2^{(j-m)(d+1)}2^{jd/p}
\sum_{J\in B_j^0}\rho_{J}(K)^{-\alpha}  .
\]
Putting the above estimates together, we estimate \eqref{firstprim} by 
\begin{multline*}
\allowdisplaybreaks
CS2^{(j-i)(d+1)}2^{jd/p}\sum_{K\in \D_j}
\rho_{I}^{-\alpha}(K)\sum_{J\in B_j^1}\rho_J(K)^{-\alpha} \\
\le CS2^{(j-i)(d+1)}
2^{-jd/p'}
\no{
\rho_{I}^{-\alpha}\sum_{J\in B_j^1}\rho_J^{-\alpha}}_1  \\
\le CS2^{(j-i)(d/p +1)}2^{jd/p} 
\sum_{J\in B_j^1} \no{
\rho_{I}^{-\alpha}1_J}_\infty \\
\le  CS2^{(j-i)(d/p +1)}2^{-id/p'} 
\sum_{J\in B_j^1} \no{
\rho_{I}^{-\alpha}1_J}_1.
\end{multline*}
Here we used $j\le i$ in the estimation of integrals.
\end{proof}

We are now ready to prove inequality \eqref{e:unif-gtheo}.
We claim that it suffices to prove inequality \eqref{e:unif-gtheo} under the additional assumption $J\in T$. 
Assume we have proven the inequality under this additional assumption.
Let $J$ be arbitrary and consider the collection $\mathcal{J}$ of maximal dyadic cubes $I$ satisfying $I\subset J$ and $I\in T$. 
Applying the assumed inequality with $J$ replaced by an element of $\mathcal{J}$ and summing
over the disjoint cubes in $\mathcal{J}$ 
we obtain the inequality for the given $J$.
Hence we can make the auxiliary assumption.

Following the decomposition \eqref{fminusg}
and splitting the sum on the left hand side of \eqref{e:unif-gtheo} along the cases $k<i$ and $i\le k$,
we may apply Lemma \ref{l:f-gI1} to the terms with $k<i$ and Lemma \ref{l:f-gI3} to the terms with $i\le k$.
Hence we are left with estimating
\begin{gather}
\label{usel23}
  CS\sum_{i\le j}\sum_{I\in T_i^0, I\subset J}\sum_{i\le k\le 0}2^{\epsilon(i-k)}\sum_{K\in B_i^0}\no{1_I \rho_K^{-\alpha}}_1  \\
\label{usel24}
+ CS\sum_{i\le j}\sum_{I\in T_i^0, I\subset J} \sum_{k<i}
 2^{\epsilon(k-i)}
\sum_{K\in B_k^0}\no{\rho_I^{-\alpha}1_K}_1,
\end{gather}
where $\epsilon=\min(\alpha-d/p,1+d/p)$,
by the right hand side of equation \eqref{e:unif-gtheo}.
We estimate \eqref{usel23} and \eqref{usel24} separately.

We first estimate \eqref{usel23} from above by
\begin{multline*}CS\sum_{i\le j}\sum_{I\in T_i^0, I\subset J}\no{1_I \sum_{K\in B_i^0}\rho_K^{-\alpha}}_1 \le CS\sum_{i\le j}\no{1_J \sum_{K\in B_i^0}\rho_K^{-\alpha}}_1 \\
\le CS\sum_{i\le j}\left(\sum_{K\in B_i^0, K\subset 3J}2^{id}+ \no{1_J\sum_{K\in B_i^{0}, K\not\subset 3J}
\rho_K^{-\alpha}}_1\right)\\
\le CS2^{jd}+ CS\sum_{i\le j}
2^{(\alpha-d)(i-j)/2}2^{jd}\le CS2^{jd}.
\end{multline*}
Then we estimate \eqref{usel24} by
\begin{multline*}
CS\sum_{k<j} \sum_{k<i\le j}\sum_{I\in T_i^0, I\subset J}
 2^{-\epsilon|k-i|}
\sum_{K\in B_k^0}\no{\rho_I^{-\alpha}1_K}_1\\
\le CS\sum_{k<j} \sum_{k<i\le j}
 2^{-\epsilon|k-i|}
\sum_{K\in B_k^0}\no{\rho_J^{-\alpha}1_K}_1\\
\le CS\sum_{k<j} 
\sum_{K\in B_k^0}\no{\rho_J^{-\alpha}1_K}_1 \le CS\no{\rho_J^{-\alpha}}_1
\le CS2^{jd}.
\end{multline*}
These are the desired estimates for
\eqref{usel23} and \eqref{usel24}. 
The proof of inequality \eqref{e:unif-gtheo} is complete.

\section{Proof of Inequality \texorpdfstring{\eqref{e:unigtheo-variant}}{(\ref{e:unigtheo-variant})}}

In this section, we assume $1\le p<\infty$. We will prove estimates with
constants independent of $p$
and thus the estimates will extend to the case $p=\infty$.

We first prove estimates for cubes that are far away from $T$.
\begin{lemma}
\label{l:offthetree1}
Let  $i \le 0$ and $I\in \D_i$ satisfy $I\not \subset 7U$.
Then
\begin{equation}\label{e:farI}   
2^{-id/p} \no{\rho_I^{-3\alpha} \psi *g}_p\le CS 2^{i\alpha}\no{1_U \rho_I^{-\alpha}}_\infty.
\end{equation}
\end{lemma}
\begin{proof}
As $g = 1_{5U}g$,
we can apply Lemma \ref{l:commutechii} 
to obtain
\[
\no{\rho_I^{-3\alpha} \psi *g}_p 
\le C \rho_I(5U)^{-2\alpha} \no{g \rho_I^{-\alpha}}_p.
\]
We use the $L^p$ bound \eqref{e:unifgnorm} on $g$, which we have already established,
and that the support of $g$ is contained in $5U$, to conclude
\[\no{g \rho_I^{-\alpha}}_p \le \no{g}_p\|1_{5U}\rho_I^{-\alpha}\|_\infty\le CS\|1_U\rho_I^{-\alpha}\|_\infty.\]
In the last inequality, we 
replaced $5U$ by $U$, thanks to the assumptions on size and location of $I$.
As $i \le 0$ and $I \not \subset 5U$,
inequality \eqref{e:farI} then follows from
$\rho_I(5U)^{-2\alpha}\le
2^{2i\alpha}\le 2^{i\alpha}2^{id/p}$.
\end{proof}

Let $\mathcal{P}$ be the collection of maximal dyadic cubes $K$
such that $K \subset 7U$ but $3K$ does not contain any cube from $T$. 
By maximality, the cubes in $\mathcal{P}$ are pairwise disjoint. They also have side-length at most $1$.

\begin{lemma}
\label{l:inK}
Let $i \le k \le 0$, $I \in \D_i$ with $I  \subset 7U$
and $\psi\in \Psi_{i-m}^{4\alpha}$.
If there exists $K\in \mathcal{P} \cap \D_k$ such that $I \subset K$,
then 
\begin{equation}\label{e:kest}  
2^{-id/p} \no{\rho_I^{-3\alpha} \psi *g}_p\le CS \left( \rho_I (( {\textstyle \frac{3}{2} } K)^{c})^{-2\alpha}+2^{i-k}\right).
\end{equation}
\end{lemma}
\begin{proof}
We write $g$ as sum of small scales and large scales, 
\begin{align}
    \label{e:gs}
g_s &=
\sum_{j=-\infty}^{k-1}\left((\psi_{j-m}*f) 1_{E_j^{1}}+
\sum_{n=1}^{d}   \partial_n^{d+1}\sigma_{n,j}\right), \\
    \label{e:gl}
g_l &= (\tau_{-m}*f)\chi + \sum_{j=k}^{0}\left((\psi_{j-m}*f) 1_{E_j^{1}}+ \sum_{n=1}^{d}  \partial_n^{d+1}\sigma_{n,j}\right).
\end{align} 
By the triangle inequality, 
it suffices to prove analogs of \eqref{e:kest} for the summands separately.

\subsection*{Small scales}
Let $F$ be the union of dyadic cubes $B \in \mathcal{D}_{k-2}$ with $\rho_B(E_{k-1}^{1}) \le 1$. 
The term corresponding to the small scales satisfies $g_s = 1_{F}g_s$, 
and we have by Lemma \ref{l:commutechii} 
\[
\no{\rho_I^{-3\alpha} (\psi * g_s) }_p
  \le C \rho_I(F) ^{-2\alpha} \no{ \rho_I^{-\alpha} g_s }_p .
\]
As the interior of $K$ is disjoint from $E_{k}^{1}$,
we conclude the interiors of ${\textstyle \frac{3}{2}} K$ and $F$ are disjoint
so that 
\[
\rho_I(F) ^{-2\alpha} \le  \rho_I( ({\textstyle \frac{3}{2} }K)^{c} ) ^{-2\alpha} .
\] 
Our estimate of $g_s$ will be complete once we show
\[
\no{ \rho_I^{-\alpha}g_s }_p 
\le 2^{id/p} S .
\]
For $h$ and $k_n$ as defined in 
\eqref{def:hpart}
and \eqref{def:knpart}, 
it holds
\begin{equation}\label{e:gsdetail}
g_s=-1_{E_{k-1}^{1}}(\tau_{k-1-m}*f)
+h+ \sum_{n=1}^d k_n .
\end{equation}
We estimate the three summands separately.
 
We start by noticing that for any $J \in \mathcal{D}_j$ with $j \le k-1$ and $J \subset E_{j}^{1}$
it holds $\rho_J(I) \ge 2$.
Suppose that such a cube $J$ is given.
If $j > i$,
then
\begin{equation*}
\rho_I(J) \ge C 2^{j-i} \rho_J(I) \ge C \sup_{y \in J} \rho_I(y)
\end{equation*} 
and for $j \le i$ 
\begin{equation}
\label{e:gap1}
\rho_I(J) \ge C \sup_{y \in J} \rho_I(y)
\end{equation}
is immediate.
We conclude that for all cubes $J$ as above 
the inequality \eqref{e:gap1} holds.

\subsection*{First term in \texorpdfstring{\eqref{e:gsdetail}}{(\ref{e:gsdetail})}}
For each cube $J\in \D_{k-1}$ with $J\subset E_{k-1}^{1}$, 
we note that 
\begin{equation}
\label{2510_first}
\|\rho_I^{-\alpha}1_J (\tau_{k-1-m}*f)\|_p^p\le
\rho_I(J)^{-\alpha p} \|1_J (\tau_{k-1-m}*f)\|^p_p\le C^{p} S^p \rho_I(J)^{-\alpha p}|J|
.\end{equation}
In the last inequality, we have estimated the $L^p$
norm using the definition of $S$ 
and that there is an cube near $J$ in $T_{k-1}^0$.
As $\rho_J(I) \ge 2$,  
we can invoke \eqref{e:gap1}.
Summing over the disjoint cubes $J \in \D_{k-1}$ with $J \subset E_{k-1}^{1}$
and using \eqref{2510_first},
we hence obtain
\[\|\rho_I^{-\alpha}1_{E_{k-1}^{1}} (\tau_{k-1-m}*f)\|_p^p \le
 C^{p} S^p  \int \rho_I(y)^{-\alpha p} \, dy
\le C^{p} S^p |I|.\]
This gives the desired bound for the first term in \eqref{e:gsdetail}.

\subsection*{Second term in \texorpdfstring{\eqref{e:gsdetail}}{(\ref{e:gsdetail})}}
We turn to the second term in \eqref{e:gsdetail}.
We use the decomposition given in Lemma \ref{l:Iispartition} 
\[
1_{E_{k-1}^{1}} = \sum_{J \in \mathcal{I}_{<k-1}} 1_{J}.
\]
Fix a cube $J\in \mathcal{I}_{j}$
with $j< k-1$. 
By Lemma \ref{l:hpart}
\[
\no{ \rho_I^{-\alpha} 1_{J} h }_p^{p}
\le C^{p} \rho_I(J)^{-\alpha p} S^{p} |J|.
\]
Using now the fact $\rho_J(I) \ge 2$ 
and its consequence \eqref{e:gap1}, 
we estimate 
\[
\sum_{J \in \mathcal{I}_{<k-1}} C^{p} \rho_I(J)^{-\alpha p} S^{p} |J|
  \le C^{p} S^{p} \int \rho_I(x)^{-\alpha p} \, dx \le C S^{p} |I|.
\]
This proves the desired estimate for $h$.

\subsection*{Third term in \texorpdfstring{\eqref{e:gsdetail}}{(\ref{e:gsdetail})}}
To bound the term $k_n$ in \eqref{e:gsdetail},
we apply \eqref{e:sigmabound} in Lemma \ref{l:sigma-only} and estimate 
\begin{multline}
\label{e:05092}
 \sum_{n=1}^{d} \no{ \rho_I^{-\alpha} k_n }_p^{p}
\le C^{p} \sum_{j=-\infty}^{k-1} \sum_{n=1}^{d} \sum_{J \in B_{j}^{1}}
\rho_I(\partial E_{j}^{1} \cap \partial J)^{-\alpha p} \no{1_{J} \partial_n^{d+1}\sigma_{n,j}}_p^{p}\\
\le C^{p} \sum_{j=-\infty}^{k-1} \sum_{J \in B_{j}^{1}}
\rho_I(\partial E_{j}^{1} \cap \partial J)^{-\alpha p} 2^{jd} S ^{p}
\end{multline}
For each $J \in B_j^{1}$ there exists $\tilde{J} \subset E_j^{1}$ with
\[
\rho_I(\partial E_{j}^{1} \cap \partial J) \ge C \rho_I(\partial E_{j}^{1} \cap \partial J \cap \partial \tilde{J}).
\]
We can then use $\rho_{\tilde{J}}(I) \ge 2$ 
to apply \eqref{e:gap1} to $\tilde{J}$ and to estimate  
\[
\rho_I(\partial E_{j}^{1} \cap \partial J) \ge C \sup_{y \in \tilde{J}} \rho_{I}(y) \ge C \sup_{y \in 4^{-1} J} \rho_{I}(y).
\]
This together with disjointness of $J \in B_{j}^{1}$ for all $j$ implies
\[
C^{p} \sum_{j=-\infty}^{k-1} \sum_{J \in B_{j}^{1}}
\rho_I(\partial E_{j}^{1} \cap \partial J)^{-\alpha p} 2^{jd}
 \le  C^{p} \int \rho_I(x)^{-\alpha p} \, dx
 \le C^{p} |I| 
\]
so that the right hand side of \eqref{e:05092} is bounded by $C^{p} |I| S^{p}$ as claimed. 
This concludes the proof of the bound for the small scales \eqref{e:gs}. 
 
\subsection*{Large scales}
We turn to estimate \eqref{e:gl}.
Consider first the sum 
\begin{equation}
\label{e:second_ibp}
\sum_{n=1}^{d} \sum_{j=k}^{0} \left((\psi_{n,j-m}*f) 1_{E_j^{1}}+   \partial_n^{d+1}\sigma_{n,j}\right).
\end{equation}
Following section \ref{sec:cone-decomposition},
we form the cone decomposition  
\[
\psi = \sum_{\nu = 1}^{d} \psi_\nu , 
\]
where the Fourier transform of $\psi_\nu$ vanishes whenever $2d |\xi_\nu|^{2} \le |\xi|^{2}$.
We notice  
\[
2^{-(d+1)(i+m)} \partial_\nu^{-d-1}  \psi_{\nu} \in C \Phi_{i-m}^{4\alpha}, \quad 2^{(d+1)(j-m)} \partial_\nu^{d+1} \psi_{n,j-m} \in C \Psi_{j-m}^{4\alpha}
\] 
where $\partial_{\nu}^{-d-1}$ is used to denote the Fourier multiplier operator with symbol $(2\pi i \xi_{\nu})^{-d-1}$.

The summands in \eqref{e:second_ibp} are $d+1$ times differentiable by Lemma \ref{l:contderivative},
and integrating by parts we obtain 
\begin{multline}
\label{e:two-sums-in-K-prop}
\nos{\rho_I^{-3\alpha} \psi_{\nu} * \left((\psi_{n,j-m}*f) 1_{E_j^{1}}+  \partial_n^{d+1}\sigma_{n,j}\right) }_p \\
= \nos{\rho_I^{-3\alpha} \partial_\nu^{-d-1}\psi_\nu * \left(( \partial_\nu^{d+1} \psi_{n,j-m}*f) 1_{E_j^{1}}+ \partial_{\nu}^{d+1} \partial_n^{d+1}\sigma_{n,j}\right) }_p \\
\le \nos{\rho_I^{-3\alpha}  \partial_\nu^{-d-1}\psi_\nu * (( \partial_\nu^{d+1} \psi_{n,j-m}*f) 1_{E_j^{1}})}_p \\
+   \nos{\rho_I^{-3\alpha}  \partial_\nu^{-d-1}\psi_\nu * \partial_{\nu}^{d+1}\partial_n^{d+1}\sigma_{n,j}}_p.
\end{multline}
By Lemma \ref{l:commutechii},
the first term is bounded by 
\[
C 2^{(d+1)(i-j)} \nos{\rho_I^{-\alpha} ( 2^{(d+1)(j-m)} \partial_\nu^{d+1}\psi_{n,j-m}*f)  }_p.
\]
Because $j \ge k$,
we can find $\widehat{I}$ with $\widehat{I} \in \mathcal{D}_j$ and $\widehat{I} \supset K \supset I$. 
We know by maximality that $9K$ contains a cube from $T$,
and hence $\rho_{\widehat{I}}(E_{j}^{0}) \le 2$.
Because $\rho_I^{-\alpha} \le \rho_{\widehat{I}}^{-\alpha}$,
we conclude that 
\begin{multline*}
\sum_{j=k}^{0}  2^{(d+1)(i-j)} C \nos{\rho_I^{-\alpha} ( 2^{(d+1)(j-m)} \partial_\nu^{d+1}\psi_{n,j-m}*f)  }_p \\
\le C 2^{(i-k)((1-1/p)d+1)+ id/p} S 
\end{multline*}
which is the claimed upper bound for the first term on the right hand side of \eqref{e:two-sums-in-K-prop}.

To estimate the second term on the right hand side of \eqref{e:two-sums-in-K-prop},
we note that by Lemma \ref{l:commutechii}
\begin{multline}
\label{e:05091}
\sum_{j=k}^{0} \nos{\rho_I^{-3\alpha}  \partial_\nu^{-d-1}\psi_n *\partial_\nu^{d+1}\partial_n^{d+1}\sigma_{n,j}}_p
\le \\
C \sum_{j=k}^{0} 2^{(d+1)(i-m)} \left(  \sum_{J \in B_{j}^{1}} 
\rho_{I}( J )^{-2\alpha p }   \no{1_{J} \partial_{\nu}^{d+1} \partial_n^{d+1}\sigma_{n,j}}_p^{p} \right)^{1/p}.
\end{multline}
Let $\hat{I} \in \mathcal{D}_j$ be the dyadic cube with $I \subset \hat{I}$.
Then $\rho_I(J) \ge \rho_{\hat{I}}(J)$.
By Lemma \ref{l:sigma-only} 
\[
\no{1_{J} \partial_{\nu}^{d+1} \partial_n^{d+1}\sigma_{n,j}}_p \le C 2^{(m-j)(d+1)} 2^{jd/p} S,
\]
and consequently we estimate the right hand side of \eqref{e:05091} by 
\begin{multline*}
C S \sum_{j=k}^{0} 2^{(d+1)(i-j)}   \no{\rho_{\hat{I}}^{-2\alpha }}_{p} 
\le C S \sum_{j=k}^{0} 2^{(d+1)(i-j)+jd/p} \\
\le C 2^{(d(1-1/p)+1)(i-k)} S 2^{id/p}
\end{multline*}
which is the desired upper bound. 

\subsection*{First term in \texorpdfstring{\eqref{e:gl}}{(\ref{e:gl})}}
The remaining term in \eqref{e:gl}
\[
\no{\rho_I^{-3\alpha} \psi  * ( (\tau_{-m}*f)\chi) }_p
\]
is estimated identically to the first term in \eqref{e:two-sums-in-K-prop}
as $\chi$ is smooth.
As we have given estimates for the contributions from each of the terms in \eqref{e:gs} and \eqref{e:gl},
the proof of the lemma is complete.
\end{proof}

\begin{lemma}
\label{l:remainder-meets-thetree}
Let $i \le 0$ and $I \in \mathcal{D}_i$. 
If $I \subset 7U$ and $I \not \subset K$ for all $K \in \mathcal{P}$, 
then there is $J \in  T_{i}$ with $\rho_I(J) \le 1$. 
\end{lemma}
\begin{proof} 
By definition, $3I$ contains an element $J$ from $T$.
Assuming $J$ is the maximal element contained in $3I$,
we see that either $J \in T_{i+1}$ and $I \subset J$ or $J \in T_{i}$ and $\rho_I(J) \le 1$.
\end{proof}

Now we can complete the proof of inequality \eqref{e:unigtheo-variant}.
If $J \cap 7U = \varnothing $,
we use Lemma \ref{l:offthetree1} to estimate the left hand side of \eqref{e:unigtheo-variant} by
\begin{equation*}
C S \sum_{i \le 0} \sup_{\substack{I \in \mathcal{D}_i \\ I \cap 7U = \varnothing \\ I \subset J}}  |I|^{\alpha /d} \no{1_U \rho_I^{-\alpha}}_\infty 
  \le C S \no{1_U \rho_J^{-\alpha}}_\infty,
\end{equation*}
which is the desired right hand side.

If $J \cap 7U \ne \varnothing$, we find a $K\in \mathcal{P}$ such that $J\subset K$ and 
we use Lemma \ref{l:inK} to estimate the left hand side of \eqref{e:unigtheo-variant} by 
\begin{equation}
\label{e:off-tree-sums}
  CS  \sum_{i\le i_0}\sup_{\substack{I \in \mathcal{D}_i \setminus T \\ I \subset J }}     \rho_I(( {\textstyle \frac{3}{2}} K)^{c})^{-2\alpha} +  (|I|/|K|) ^{1/d}    \le CS.
\end{equation}

This concludes the proof.

\section{Proof of Proposition \ref{t:spq}}
Assume first $1\le p\le q\le \infty$.  
By H\"older's inequality,
we have
\[\|\rho_I^{-\alpha(1+\frac{q-p}{qp})}
\phi*f\|_p\le \|\rho_I^{-\alpha\frac {q-p}{qp}}\|_{\frac{qp}{q-p}}
\|\rho_I^{-\alpha}
\phi*f\|_q,\]
and thus for some universal constant $C$ 
\[ S_{\alpha(1+\frac {q-p}{qp}),p}\le CS_{\alpha,q},\] 
which proves \eqref{e:holder}.
 
Assume now $1\le q\le p\le \infty$.
We  use logarithmic convexity
\[S_{\alpha,p}\le S_{\alpha,q}^{q/p}S_{\alpha,\infty} ^{1-q/p}.\]
To bound $S_{\alpha,\infty}$, we proceed similarly as in Lemma \ref{l:commutechii}
and note that for $i\le i_0$, $I\in D_i\cap M_U$, and
$\phi\in \Phi^{4\alpha}_{i-m-2}$,
\[ \|\rho_I^{-\alpha}(\phi*f)\|_\infty\le 
\|\rho_I^{-\alpha}(\phi*\rho_U^{\alpha})\|_\infty \|\rho_U^{-\alpha} f\|_\infty\]
and
\[\|\rho_I^{-\alpha}(\phi*\rho_U^{\alpha})\|_\infty\le 
C \|\rho_I^{-\alpha} \rho_U^{\alpha}\|_\infty\le C.\]
This proves  \eqref{e:logconvex}.
 
Finally,
if one is willing to lose uniformity in the parameter $m$, 
one can use a local Bernstein's inequality.
Let $\varphi$ be a Schwartz function on $\R^d$ 
so that $\widehat{\varphi}(\xi)=0$ for $|\xi| > 1$
and $\varphi(x) \ge c_d > 0$ for $x \in U$ for a dimensional constant $c_d$.
Denote
$\varphi_J(x)=\varphi(2^{-j} (x-c(J)))$
for $j\in \Z$ and $J\in \D_j$ and $c(J)$ the center of $J$.
Then for $I\in \D_i$ and $\phi \in \Phi_{i-m-2}^{4\alpha}$
\begin{multline*}\|\rho_{I}^{-\alpha} \phi*f\|_\infty
 \le 
  C \sum_{J\in \D_i} \rho_I(J)^{-\alpha} \| \varphi_J \phi*f\|_\infty\\ 
\le C 2^{d(m-i)}\sum_{J\in \D_i} \rho_I(J)^{-\alpha} 
 \| \varphi_J \phi*f\|_1
 \le C 2^{d(m-i)} \|\rho_I^{-\alpha} \phi*f\|_1.
\end{multline*}
Here we used that the Fourier transform of $\varphi_J \phi*f $ is supported in a ball of radius $2^{3+m-i}$ and Bernstein's inequality.
This completes the proof of \eqref{e:bernstein}.
 
\bibliography{sample}

\bibliographystyle{abbrv}

\end{document}